\documentclass[a4paper,12pt,reqno]{amsart}

\usepackage{amsmath}
\usepackage{amssymb}
\usepackage{amsfonts}
\usepackage{graphicx}
\usepackage[colorlinks]{hyperref}
\renewcommand\eqref[1]{(\ref{#1})} 

\graphicspath{ {images/} }
\title[Weighted anisotropic Hardy and Rellich type inequalities ]{Weighted anisotropic Hardy and Rellich type inequalities for general vector fields}
\author[Michael Ruzhansky]{Michael Ruzhansky}
\address{Michael Ruzhansky:
	\endgraf
	Department of Mathematics
	\endgraf
	Imperial College London
	\endgraf
	180 Queen’s Gate, London, SW7 2AZ 
	\endgraf
	United Kingdom
	\endgraf
	{\it E-mail address} {\rm m.ruzhansky@imperial.ac.uk}
}

\author[Bolys Sabitbek]{Bolys Sabitbek}
\address{ Bolys Sabitbek:
	\endgraf
	Institute of Mathematics and Mathematical Modeling 
	\endgraf
	 125 Pushkin Street., Almaty, 050010
	 \endgraf
	 Kazakhstan
	 \endgraf
	 and
	 \endgraf 
	 Department of Mechanics and Mathematics
	 \endgraf 
	 Al-Farabi Kazakh National University 
	 \endgraf
	 71 al-Farabi Ave., Almaty, 050040 
	 \endgraf Kazakhstan
	 \endgraf
	{\it E-mail address} {\rm b.sabitbek@math.kz}
}

\author[Durvudkhan Suragan]{Durvudkhan Suragan}
\address{
	Durvudkhan Suragan:
	\endgraf
	Department of Mathematics
	\endgraf
	Nazarbayev University
	\endgraf
	53 Kabanbay batyr Ave., Astana, 010000
	\endgraf
	Kazakhstan
	\endgraf
	{\it E-mail address} {\rm durvudkhan.suragan@nu.edu.kz}
}


\thanks{The first
 author was supported by the EPSRC Grant 
EP/R003025/1 and by 
 the Leverhulme Research Grant RPG-2017-151. The second author was supported by the MESRK target program BR05236656. The third author was supported in parts by the MESRK grant AP05130981. No new data was collected or generated during the course of this research.}

\newtheoremstyle{theorem}
{10pt}          
{10pt}  
{\sl}  
{\parindent}     
{\bf}  
{. }    
{ }    
{}     
\theoremstyle{theorem}

\numberwithin{equation}{section}
\theoremstyle{plain}
\newtheorem{thm}{Theorem}[section]
\newtheorem{prop}[thm]{Proposition}
\newtheorem{cor}[thm]{Corollary}
\newtheorem{lem}[thm]{Lemma}

\theoremstyle{definition}

\newtheorem{rem}[thm]{Remark}

\newcommand{\G}{\mathbb{G}}

\newtheoremstyle{defi}
{10pt}          
{10pt}  
{\rm}  
{\parindent}     
{\bf}  
{. }    
{ }    
{}     
\theoremstyle{defi}

\begin{document}
		\begin{abstract}
		In this paper, we establish the weighted anisotropic Hardy and Rellich type inequalities with boundary terms for general (real-valued) vector fields. As consequences, we derive new as well as many of the fundamental Hardy and Rellich type inequalities which are known in different settings.     
	\end{abstract}
	
	\maketitle
	\section{Introduction}
	
	The main aim of this paper is to present the weighted anisotropic Hardy and Rellich type inequalities with boundary terms for general (real-valued) vector fields. The consequences recover many previously known results in different settings. The anisotropic Picone type identities play key roles in our proofs.  
	  
	Recall the Hardy inequality for $\Omega \subset \mathbb{R}^n$ stating that
	\begin{equation}
		\int_{\Omega} |\nabla u|^p dx \geq C \int_{\Omega} \frac{|u|^p}{|x|^p} dx, \,\, u \in C^1_0(\Omega),
	\end{equation}
	where $\nabla$ is the Euclidean gradient and $p>1$. It has been vastly studied by many authors and developed in different settings, see e.g. \cite{BFT_Hardy}, \cite{Brezis_Marcus}, \cite{DAmbrosio_Heisenberg}, \cite{Davies_Hinz}, \cite{GKY_Hardy}  and the references therein.  
	
First, let us review some of the recent results:
\begin{itemize}
	\item Hardy type inequalities in the setting of the {\em Heisenberg group} $\mathbb{H}^n$ have the following form
	\begin{equation}\label{H1}
		\int_{\mathbb{H}^n} |\nabla_{H} u|^2 dx \geq C \int_{\mathbb{H}^n} \frac{\psi_H^2}{\rho^2} |u|^2 dx, \,\, u \in C_0^1 (\mathbb{H}^n \backslash \{0\}), 
	\end{equation}
	where $\nabla_{H}$ is a (horizontal) gradient associated to the sub-Laplacian, $\psi_H$ and $\rho$ are a weight function and a suitable distance from the origin, respectively. For example, Garofalo and Lanconelli in \cite{Gar_Lan}, Niu, Zhang and Wang in \cite{NZW}, D'Ambrosio in \cite{DAmbrosio_Hardy} and others have made a contribution to prove the above inequality and its extensions in $\mathbb{H}^n$.
	\item Hardy type inequalities in the setting of  the {\em Carnot group} $\G$ can be given by the formula
	\begin{equation}\label{H2}
		\int_{\G} d^{\alpha} |\nabla_{H} u|^2 dx \geq C \int_{\G} d^{\alpha-2}|\nabla_{H} d|^{2}|u|^2dx, \,\, u \in C_0^{\infty}(\G \backslash \{0\}),
	\end{equation}
	where $\nabla_{H}$ is the horizontal gradient on $\G$, $\alpha \in \mathbb{R}$, and $d$ is a homogeneous norm associated with a fundamental solution for the sub-Laplacian. For instance, the Hardy type inequalities on $\G$ have been studied by Goldstein, Kombe and Yener in \cite{GKY_Hardy}, Kombe in \cite{Kombe2010}, Wang and Niu in \cite{Wang_Niu} and the authors in \cite{RS17a}.
	\item Hardy type inequalities in the setting of {\em general vector fields} can be presented in the form
	\begin{equation}\label{H3}
		\int_{\Omega} |\nabla_X u|^p dx \geq C\int_{\Omega} \frac{|\nabla_X \phi|^p}{\phi^p} |u|^p dx, \,\, u \in C_0^1(\Omega),
	\end{equation}
	where $\nabla_X:=(X_1,\ldots,X_N)$ and $\phi$ is any positive weight function. To the best of our knowledge, D'Ambrosio  obtained first versions of Hardy type inequalities for general vector fields in \cite{DAmbrosio_Hardy}.
\end{itemize}
	
	Consider a family of real vector fields $\{X_k\}_{k=1}^N, \, N\leq n,$ on a smooth manifold $M$ with dimension $n$ and a volume form $dx$.	
	In Section \ref{sec1-1}, we use the approach developed in \cite{GKY_Hardy} and \cite{RSS_aniso} to establish the following weighted Hardy type inequalities for general vector fields 
	\begin{equation*}
		\int_{\Omega} W(x) |\nabla_X u|^p dx \geq \int_{\Omega} H(x) |u|^p dx, \,\,u \in C_0^1(\Omega),
	\end{equation*}
	with the hypothesis
	\begin{equation*}
		-\nabla_X \cdot (W(x)|\nabla_X v|^{p-2}\nabla_X v) \geq H(x)v^{p-1},
	\end{equation*}
	where $\nabla_X=(X_{1},X_{2},\ldots,X_{N})$ is the associated gradient and $v$ is a function satisfying the above hypothesis. 
	From this  weighted Hardy type inequality, we recover most of the fundamental Hardy type inequalities including \eqref{H1}, \eqref{H2} and \eqref{H3}. In Section \ref{sec1-2}, 
we prove the weighted anisotropic Rellich type inequality for general vector fields.

	\section{Weighted anisotropic Hardy type inequality}
	\label{sec1-1}
	In this section, we obtain the weighted anisotropic Hardy type inequalities for general (real-valued) vector fields. It will be proved by using the anisotropic Picone type identity. As consequences, we discover most of the Hardy type inequalities and the uncertainty principles which are known in the setting of the Euclidean space, Heisenberg and Carnot groups. 
	
	Consider a family of real vector fields $\{X_k\}_{k=1}^N, \, N \leq n,$ on a smooth manifold $M$ with dimension $n$ and a volume form $dx$. 
		 Then we say that an open bounded set $\Omega \subset M$ is an admissible domain if its boundary $\partial \Omega$ has no self-intersections, and if the vector fields $\{X_k\}^N_{k=1}$ satisfy 
		\begin{equation}\label{1}
		\sum_{k=1}^{N} \int_{\Omega} X_kf_k dx = \sum_{k=1}^{N} \int_{\partial \Omega} f_k \langle X_k, dx\rangle,
		\end{equation}
		for all $f_k \in C^1(\Omega)\cap C(\overline{\Omega}), k=1,\ldots,N$. For more details see \cite{RS17b} and \cite{RS_local}, where it is shown that \eqref{1} is usually satisfied for most domains $\Omega$, and we also recall some examples below.
		
		First, we formulate an assumption which is important for presenting some examples of Theorem \ref{THMgeneralHardy} and of other related results:
		\begin{description}
			\item[Assumption] Let $T_y \subset M$ be an open set containing $y \in M$ such that the operator 
			$$\mathcal{L}:= \sum_{i=1}^{N} X_i^2$$ has a fundamental solution in $T_y$, that is, there exists a function $\Gamma_y \in C^2(T_y \backslash \{y\})$ such that
			\begin{equation}
			- \mathcal{L} \Gamma_y = \delta_y \,\, \text{in} \,\, T_y,
			\end{equation} 
			where $\delta_y$ is the Dirac $\delta$-distribution at $y$.
		\end{description}
	
	We will say that an admissible domain $\Omega$ is a strongly admissible domain with $y \in M$ if the above assumption is satisfied, $\Omega \subset T_y$, and \eqref{1} holds for $f_k=vX_k \Gamma_y$ for all $v \in C^1(\Omega)\cap C(\overline{\Omega})$.
	 
		Note that the fundamental solution for sums of squares of vector fields satisfying H\"ormander's condition were obtained by S\'anchez-Calle in \cite{Sanchez-Calle}.
		
		Let us recall several important examples from \cite{RS_local} which satisfy the above condition:
		\begin{description}
			\item[Example 1] Let $M$ be a stratified group with $n \geq 3$, and let $\{X_k\}_{k=1}^N$ be the left-invariant vector fields giving the first stratum of $M$. Then any open bounded set $\Omega \subset M$ with a piecewise smooth simple boundary is strongly admissible. 
			\item[Example 2] Let $M \equiv \mathbb{R}^n$ with $n\geq 3$, and let the vector fields $X_k$ with $k=1,\ldots,N$, $N\leq n$, have the following form
			\begin{equation}\label{Xk}
				X_k := \frac{\partial }{\partial x_k} + \sum_{m=N+1}^{n}a_{k,m}(x)\frac{\partial }{\partial x_m},
			\end{equation}
			where $a_{k,m}(x)$ are locally $C^{1,\alpha}$-regular for some $0<\alpha\leq 1$, where $C^{1,\alpha}$ stands for the space of functions with $X_k$-derivative in the H\"older space $C^{\alpha}$ with respect ot the control distance defined by these vector fields. Assume that 
			\begin{equation*}
				\frac{\partial }{\partial x_k} = \sum_{1\leq i>j\leq N} \lambda^{i,j}(x)[X_i,X_j]
			\end{equation*}
			for all $k=N+1,\ldots,n$ with $\lambda_k^{i,j} \in L_{loc}^{\infty}(M)$.  Then any open bounded set $\Omega \subset M \equiv \mathbb{R}^n$ with a piecewise smooth simple boundary is strongly admissible.
			\item[Example 3] More generally, let $M\equiv\mathbb{R}^n$ with $n\geq 3$. Let the vector fields $X_k$ for $k=1,\ldots,N$, $N\leq n$, satisfy the H\"ormander commutator condition of step $r\geq 2$. Assume that all the vector fields $X_k$ for $k=1,\ldots,N$ belong to $C^{r,\alpha}(U)$ for some $0<\alpha \leq 1$ and $U \subset M \equiv \mathbb{R}^n$, and if $r=2$, then we assume $\alpha=1$. Then if $X_k$'s are in the form \eqref{Xk}, then any open bounded set $\Omega \subset M \equiv \mathbb{R}^n$ with a piecewise smooth simple boundary is strongly admissible.
		\end{description}
Moreover, let us recall the following analogue of Green's formulae which was proved in \cite{RS_local}.
\begin{prop}[Green's formulae]\label{Green}
	Let $\Omega \subset M$ be an admissible domain. Let $u \in C^2(\Omega)\cap C^1(\overline{\Omega})$ and $v \in C^1(\Omega)\cap C(\overline{\Omega})$, then we have the following analogue of Green's first formula 
	\begin{equation}
		\int_{\G} \left( (\widetilde{\nabla} v)u + v \mathcal{L} u \right) dx = \int_{\partial \Omega} v \langle \widetilde{\nabla} u, dx \rangle,
	\end{equation}
	where
	\begin{equation}
		\widetilde{\nabla} u = \sum_{i=1}^{N} (X_iu)X_i.
	\end{equation}
	If $u,v \in C^2(\Omega)\cap C^1(\overline{\Omega})$, then we have the following analogue of Green's second formula 
	 \begin{equation}
	 	\int_{\Omega} (u\mathcal{L}v-v\mathcal{L}u)dx = \int_{\partial \Omega} \left( u\langle \widetilde{\nabla} v,dx \rangle -v\langle \widetilde{\nabla} u,dx\rangle \right).
	 \end{equation}
\end{prop} 
	\subsection{Anisotropic Picone type identity}
	First, we present the anisotropic Picone type identity for vector fields. 
\begin{lem}\label{Picone}
	Let $\Omega \subset M$ be an open set. Let $u,v$ be differentiable a.e. in $\Omega$, $v>0$ a.e. in $\Omega$ and $u\geq 0$. Define
	\begin{equation}\label{R}
			R(u,v) := \sum_{i=1}^{N} \left| X_i u\right|^{p_i} - \sum_{i=1}^{N} X_i \left(\frac{u^{p_i}}{v^{p_i-1}}\right)\left|X_i v \right|^{p_i-2}X_iv,
	\end{equation}
	\begin{align}\label{L}
		L(u,v) := \sum_{i=1}^{N} \left| X_i u \right|^{p_i} &- \sum_{i=1}^{N}p_i\frac{u^{p_i-1}}{v^{p_i-1}} \left|X_i v\right|^{p_i-2} X_ivX_i u \nonumber\\
		&+ \sum_{i=1}^{N} (p_i-1)\frac{u^{p_i}}{v^{p_i}} \left|X_i v\right|^{p_i},
	\end{align}
	where $p_i>1$, $i=1,\ldots,N$. Then
	\begin{equation}\label{1.3}
		L(u,v)=R(u,v) \geq 0.
	\end{equation}
In addition, we have $L(u,v)=0$ a.e. in $\Omega$ if and only if $u=cv$ a.e. in $\Omega$ with a positive constant $c$.
\end{lem}
Note that Lemma \ref{Picone} for the left-invariant vector fields in the setting of stratified groups was proved in \cite{RSS_aniso}.
	\begin{proof}[Proof of Lemma \ref{Picone}]
		First, we show the equality in \eqref{1.3} by a direct computation as follows
		\begin{align*}
			R(u,v) &= \sum_{i=1}^{N} \left|X_i u\right|^{p_i} - \sum_{i=1}^{N} X_i \left( \frac{u^{p_i}}{v^{p_i-1}}\right) |X_i v|^{p_i-2}X_i v 
			\\
			&= \sum_{i=1}^{N} \left|X_i u\right|^{p_i} - \sum_{i=1}^{N}p_i\frac{u^{p_i-1}}{v^{p_i-1}}|X_iv|^{p_i-2}X_i v X_iu +\sum_{i=1}^{N} (p_i-1)\frac{u^{p_i}}{v^{p_i}}|X_iv|^{p_i}\\
			& = L(u,v).
		\end{align*}
		Now we rewrite $L(u,v)$ to see $L(u,v)\geq 0$, that is,
		\begin{align*}
			L(u,v) = & \sum_{i=1}^{N} |X_i u|^{p_i} - \sum_{i=1}^{N} p_i \frac{u^{p_i-1}}{v^{p_i-1}} |X_i v|^{p_i-1}|X_i u| + \sum_{i=1}^{N}(p_i-1)\frac{u^{p_i}}{v^{p_i}}|X_i v|^{p_i} \\
			+&\sum_{i=1}^{N}p_i \frac{u^{p_i-1}}{v^{p_i-1}} |X_i v|^{p_i-2} \left(|X_i v||X_i u| -X_i v X_iu  \right) \\
			=& S_1 +S_2,
		\end{align*}
		where we denote $S_1$ and $S_2$ in the following form
		\begin{align*}
			S_1:=& \sum_{i=1}^{N} p_i \left[ \frac{1}{p_i} |X_i u|^{p_i} + \frac{p_i-1}{p_i}\left(\left(\frac{u}{v}|X_i v|\right)^{p_i-1}\right)^{\frac{p_i}{p_i-1}}\right] \\
			 -& \sum_{i=1}^{N}p_i \frac{u^{p_i-1}}{v^{p_i-1}}|X_i v|^{p_i-1}|X_i u|,
		\end{align*}
		and
		\begin{equation*}
			S_2:= \sum_{i=1}^{N}p_i \frac{u^{p_i-1}}{v^{p_i-1}}|X_i v|^{p_i-2}\left(|X_i v||X_i u| -X_i v X_iu\right).
		\end{equation*}
	Since $|X_iv||X_i u| \geq X_i v X_iu$ we have $S_2\geq 0$. To check that $S_1 \geq 0$ we will use Young's inequality for $a\geq 0$ and $b \geq 0$ stating that
		\begin{equation}\label{2.2}
			ab \leq \frac{a^{p_i}}{p_i} + \frac{b^{q_i}}{q_i},
		\end{equation}
		where $p_i>1, q_i>1$, and $\frac{1}{p_i}+\frac{1}{q_i}=1$ for $i=1,\ldots,N$. The equality in \eqref{2.2} holds if and only if $a^{p_i}=b^{q_i}$, i.e. if $a = b^{\frac{1}{p_i-1}}$. By setting  $a= |X_i u|$ and $b =\left(\frac{u}{v}|X_iv|\right)^{p_i-1}$ in \eqref{2.2}, we get
		\begin{equation}\label{2.3}
			p_i |X_i u|\left(\frac{u}{v}|X_i v|\right)^{p_i-1} \leq p_i \left[\frac{1}{p_i} |X_i u|^{p_i} + \frac{p_i-1}{p_i}\left(\left(\frac{u}{v}|X_i v|\right)^{p_i-1}\right)^{\frac{p_i}{p_i-1}}\right].
		\end{equation}
		This yields $S_1 \geq 0$ which proves that $L(u,v)=S_1+S_2 \geq 0$.
		It is easy to check that $u=cv$ implies $R(u,v)=0$. Now let us show that $L(u,v)=0$ implies $u=cv$. Due to $u(x)\geq 0$ and $L(u,v)(x_0)=0, \, x_0\in \Omega,$ we consider the two cases $u(x_0)>0$ and $u(x_0) =0.$ 
	 For the case $u(x_0)>0$ we conclude from $L(u,v)(x_0)=0$ that $S_1=0$ and $S_2=0$. Then $S_1=0$ implies
		\begin{equation}\label{2.4}
			|X_i u| = \frac{u}{v} |X_i v|, \quad i=1,\ldots,N,
		\end{equation}
		and $S_2=0$ implies
		\begin{equation}\label{eee}
			|X_i v| |X_i u| - X_i v X_i u = 0, \quad i=1,\ldots,N.
		\end{equation}
		The combination of \eqref{2.4} and \eqref{eee} gives
		\begin{equation}
			\frac{X_i u}{X_i v} = \frac{u }{v}=c, \quad \text{with} \quad c\neq 0, \quad i=1,\ldots,N .
		\end{equation}

 Let us denote $\Omega^*:=\{x \in \Omega | u(x)=0 \}$. If $\Omega^* \neq \Omega$, then suppose that $x_0 \in \partial \Omega^*$. Then there exists a sequence $x_k \notin \Omega^*$ such that $x_k \rightarrow x_0$. In particular, $u(x_k) \neq 0$, and hence by the first case we have $u(x_k)=cv(x_k)$. Passing to the limit we get $u(x_0) = c v(x_0)$. Since $u(x_0)= 0$ and $v(x_0)\neq 0$, we get that $c=0$. But then by the first case again, since $u =cv$ and $u\neq 0$ in $\Omega \backslash \Omega^*$, it is impossible that $c=0$. This contradiction implies that $\Omega^* = \Omega$.  
The proof of Lemma \ref{Picone} is complete.	
	\end{proof}

\subsection{Weighted anisotropic Hardy type inequality}
Now we are ready to obtain the weighted anisotropic Hardy type inequalities for general vector fields by using the anisotropic Picone type identity.
	\begin{thm}\label{THMgeneralHardy}
		Let $\Omega \subset M$ be an admissible domain. Let $W_i(x)\geq 0$ and $H_i(x)\geq 0$ be functions with $i=1,\ldots,N$, such that for a function $v \in C^1(\Omega)\bigcap C(\overline{\Omega})$ and $v >0$ a.e. in $\Omega$, we have
		\begin{equation}\label{3.1}
			- X_i (W_i(x)|X_i v|^{p_i-2}X_i v)\geq H_i(x) v^{p_i-1},\quad i=1,\ldots,N.
		\end{equation}
	Then, for all functions $0 \leq u  \in C^2(\Omega)\bigcap C^1(\overline{\Omega})$ and the positive function $v \in C^1(\Omega)\bigcap C(\overline{\Omega})$ satisfying \eqref{3.1}, we get
		\begin{align}\label{3.2}
			\sum_{i=1}^{N}\int_{\Omega} W_i(x) |X_i u|^{p_i} dx &\geq 	\sum_{i=1}^{N}  \int_{\Omega} H_i(x) |u|^{p_i} dx \\
			& + \sum_{i=1}^{N} \int_{\partial \Omega} \frac{u^{p_i}}{v^{p_i-1}} \langle \widetilde{\nabla}_i \left(W_i(x)|X_i v|^{p_i-2}X_iv\right), dx \rangle \nonumber,
		\end{align}
	where $ \widetilde{\nabla}_i f = X_ifX_i$ and $p_i>1$, for $i=1,\ldots,N$.	
	\end{thm}
\begin{rem}
	Note that if $u$ vanishes on the boundary $\partial \Omega$ and $p_i=p$, then we have the weighted Hardy type inequalities for general vector fields 
	\begin{align}\label{Hardy_without}
	\int_{\Omega} W(x) |\nabla_X u|^{p} dx \geq  \int_{\Omega} H(x) |u|^{p} dx, 
	\end{align}
	where $\nabla_X := (X_1,\ldots,X_N)$. 
	
	A Carnot group version of this Hardy type inequality was obtained with the slightly different proof by Goldstein, Kombe and Yener in \cite{GKY_Hardy}.
\end{rem}
\begin{proof}[Proof of Theorem \ref{THMgeneralHardy}.]
	Let us give a brief outline of the following proof. We start by using the property of the anisotropic Picone type identity \eqref{1.3}, then we apply the divergence theorem and the hypothesis \eqref{3.1}, respectively. At the end, we arrive at \eqref{3.2}. Thus, we have
	\begin{align*}
		0 \leq& \int_{\Omega} \sum_{i=1}^{N}W_i(x)L(u,v) dx = \int_{\Omega} \sum_{i=1}^{N}W_i(x)R(u,v) dx \\
		 =& \sum_{i=1}^{N} \int_{\Omega}W_i(x) |X_i u|^{p_i} dx - 	\sum_{i=1}^{N} \int_{\Omega} X_i \left(\frac{u^{p_i}}{v^{p_i-1}}\right)W_i(x)|X_i v|^{p_i-2}X_iv dx\\
		 =& \sum_{i=1}^{N} \int_{\Omega}W_i(x) |X_i u|^{p_i} dx 
		+ 	\sum_{i=1}^{N} \int_{\Omega} \frac{u^{p_i}}{v^{p_i-1}} X_i\left(W_i(x)|X_i v|^{p_i-2}X_iv\right) dx\\
		 -& \sum_{i=1}^{N} \int_{\partial \Omega} \frac{u^{p_i}}{v^{p_i-1}} \langle \widetilde{\nabla}_i \left(W_i(x)|X_i v|^{p_i-2}X_iv\right), dx \rangle \\
		 \leq& \sum_{i=1}^{N} \int_{\Omega}W_i(x) |X_i u|^{p_i} dx -  \sum_{i=1}^{N} \int_{\Omega} H_i(x) u^{p_i} dx \\
		 -& \sum_{i=1}^{N} \int_{\partial \Omega} \frac{u^{p_i}}{v^{p_i-1}} \langle \widetilde{\nabla}_i \left(W_i(x)|X_i v|^{p_i-2}X_iv\right), dx \rangle, 
	\end{align*}
	where $\widetilde{\nabla}_i f = X_ifX_i$. This completes the proof of Theorem \ref{THMgeneralHardy}.
\end{proof}
\subsection{Consequences of the weighted anisotropic Hardy type inequalities}
Now we present some concrete examples of the weighted anisotropic Hardy type inequalities \eqref{3.2}.

Note that examples of the weighted anisotropic Hardy type inequalities on $M$ will be expressed in terms of the fundamental solution $\Gamma=\Gamma_y(x)$ in the \textbf{assumption}. For brevity, we can just write it as $\Gamma$, if we fix some $y \in M$ and the corresponding $T_y$ and $\Gamma_y$.

\begin{cor}\label{cor1}
	Let $\Omega \subset M$ be an admissible domain. Let $\alpha \in \mathbb{R}, 1<p_i<\beta+\alpha, i=1,\ldots,N,$ and $\gamma >-1, \beta>2$. Then for all $u \in C_0^{\infty}(\Omega \backslash \{0\})$ we have
	\begin{equation}\label{ex1}
		\sum_{i=1}^{N}\int_{\Omega} \Gamma^{\frac{\alpha}{2-\beta}}|X_i\Gamma^{\frac{1}{2-\beta}}|^{\gamma}|X_iu|^{p_i} dx \geq\sum_{i=1}^{N} \left(\frac{\beta+\alpha-p_i}{p_i}\right)^{p_i}\int_{\Omega} \Gamma^{\frac{\alpha-p_i}{2-\beta}}|X_i\Gamma^{\frac{1}{2-\beta}}|^{p_i+\gamma}|u|^{p_i}dx.
	\end{equation}
\end{cor}
Note that \eqref{ex1} is an analogue of the result of Wang and Niu \cite{Wang_Niu}, now for general vector fields.
\begin{rem}
	By taking $\gamma=0$ and $p_i=2$ we have the following inequality 
	\begin{equation}\label{4}
\int_{\Omega} \Gamma^{\frac{\alpha}{2-\beta}}|\nabla_X u|^{2} dx \geq\sum_{i=1}^{N} \left(\frac{\beta+\alpha-2}{2}\right)^{2}\int_{\Omega} \Gamma^{\frac{\alpha-2}{2-\beta}}|\nabla_X \Gamma^{\frac{1}{2-\beta}}|^{2}|u|^{2}dx,
	\end{equation}
	for all $u \in C_0^{\infty}(\Omega)$ and where $\nabla_X =(X_1,\ldots,X_N)$. 
	
	Inequality \eqref{4} for general vector fields was established in \cite{RS_local}.
\end{rem}
\begin{proof}[Proof of Corollary \ref{cor1}]
	Consider the functions $W_i$ and $v$ such that 
	\begin{equation}\label{W_1}
	W_i = d^{\alpha}|X_id|^{\gamma} \,\, \text{and} \,\, v = \Gamma^{\frac{\psi}{2-\beta}}=d^{\psi},
	\end{equation}
	where we denote $d=\Gamma^{\frac{1}{2-\beta}}$ and $\psi = - \left(\frac{\beta+\alpha-p_i}{p_i}\right)$ for simplicity.
	Now we plug \eqref{W_1} in \eqref{3.1} to calculate the function $H_i$.
	Before we need to have the following computations  
	\begin{align*}
		X_iv &= \psi d^{\psi-1}X_id, \\
		|X_iv|^{p_i-2}&=|\psi|^{p_i-2}d^{(\psi-1)(p_i-2)}|X_id|^{p_i-2},\\
		W_i|X_iv|^{p_i-2}X_iv &= |\psi|^{p_i-2}\psi d^{\alpha+(\psi-1)(p_i-1)}|X_id|^{\gamma +p_i-2}X_id. 
	\end{align*}
	Also, we get
		\begin{align}\label{X2d}
\sum_{i=1}^{N}	X_i^2 d^{\alpha} &= \sum_{i=1}^{N}X_i(X_i \Gamma^{\frac{\alpha}{2-\beta}}) = \sum_{i=1}^{N}X_i\left(\frac{\alpha}{2-\beta} \Gamma^{\frac{\alpha +\beta-2}{2-\beta}}X_i\Gamma \right) \nonumber \\
	& = \frac{\alpha(\alpha+\beta-2)}{(2-\beta)^2}\Gamma^{\frac{\alpha +2\beta-4}{2-\beta}}\sum_{i=1}^{N}|X_i\Gamma|^2 + \frac{\alpha}{2-\beta} \Gamma^{\frac{\alpha +\beta-2}{2-\beta}}\sum_{i=1}^{N}X_i^2\Gamma \nonumber\\
	& =\frac{\alpha(\alpha+\beta-2)}{(2-\beta)^2} d^{\alpha +2\beta-4}\sum_{i=1}^{N}|X_id^{2-\beta}|^2 \nonumber\\
	& = \alpha(\alpha+\beta-2)d^{\alpha-2}\sum_{i=1}^{N}|X_id|^2. 
	\end{align}
	We observe that $\sum_{i=1}^{N}X_i^2 \Gamma=0$, since $\Gamma=\Gamma_y$ is the fundamental solution to $\mathcal{L}$. Also, we have
	\begin{align}\label{X2d2}
		X_i|X_id|^{\gamma} &= X_i((X_id)^2)^{\frac{\gamma}{2}} \nonumber\\
		&=\gamma |X_id|^{\gamma-2}X_id X_i^2 d \nonumber\\
		& =\gamma(\beta-1)d^{-1}|X_id|^{\gamma}X_id. 
	\end{align}
	In the last line, we have used \eqref{X2d} with $\alpha=1$.
	
	Using \eqref{X2d} and \eqref{X2d2} we compute 
	\begin{align*}
		X_i(W_i|X_iv|^{p_i-2}X_iv) &= |\psi|^{p_i-2}\psi X_i\left(\ d^{\alpha+(\psi-1)(p_i-1)}|X_id|^{\gamma +p_i-2}X_id\right)\\
		& =|\psi|^{p_i-2}\psi \left( (\alpha+(\psi-1)(p_i-1)) d^{\alpha+(\psi-1)(p_i-1)-1}|X_id|^{\gamma +p_i} \right)\\
		& + |\psi|^{p_i-2}\psi \left((\gamma +p_i-2)(\beta-1) d^{\alpha+(\psi-1)(p_i-1)-1}|X_id|^{\gamma +p_i}  \right)\\
		& + |\psi|^{p_i-2}\psi \left( (\beta-1) d^{\alpha+(\psi-1)(p_i-1)-1}|X_id|^{\gamma +p_i} \right) \\
		& = |\psi|^{p_i-2}\psi   \left( -\psi  + (\gamma +p_i-2)(\beta-1)  \right)d^{\alpha-p_i+ \psi(p_i-1)}|X_id|^{\gamma +p_i} \\
		& =- |\psi|^{p_i} d^{\alpha-p_i}|X_id|^{\gamma +p_i} v^{p_i-1} \\
		& +|\psi|^{p_i-2}\psi(\gamma +p_i-2)(\beta-1) d^{\alpha-p_i}|X_id|^{\gamma +p_i} v^{p_i-1}.
	\end{align*}
	Now we put back the value of $\psi$, then we get
	\begin{align*}
			 &-X_i(W_i|X_iv|^{p_i-2}X_iv) = \left|\frac{\beta+\alpha-p_i}{p_i}\right|^{p_i} d^{\alpha-p_i}|X_id|^{\gamma +p_i} v^{p_i-1} \\
			 & +\left|\frac{\beta+\alpha-p_i}{p_i}\right|^{p_i-2} \left(\frac{\beta+\alpha-p_i}{p_i}\right)(\gamma +p_i-2)(\beta-1)   d^{\alpha-p_i}|X_id|^{\gamma +p_i} v^{p_i-1} \\
			 &\geq \left|\frac{\beta+\alpha-p_i}{p_i}\right|^{p_i} d^{\alpha-p_i}|X_id|^{\gamma +p_i} v^{p_i-1}\\
			 & \geq H_i(x)v^{p_i-1}.
	\end{align*}
	So we have satisfied the hypothesis, then we plug the values of functions $W_i$ and $$H_i=\left|\frac{\beta+\alpha-p_i}{p_i}\right|^{p_i} \Gamma^{\frac{\alpha-p_i}{2-\beta}}|X_i\Gamma^{\frac{1}{2-\beta}}|^{\gamma +p_i},$$ in \eqref{3.2}, which completes the proof. 
\end{proof}
\begin{cor}\label{cor0}
		Let $\Omega \subset M$ be an admissible domain. Let $\alpha, \gamma \in \mathbb{R}$ and $\alpha\neq 0, \beta>2$. Then for any $u \in C_0^1(\Omega)$ we have
		\begin{equation}
			\sum_{i=1}^{N} \int_{\Omega} \Gamma^{\frac{\gamma+p_i}{2-\beta}} |X_i u|^{p_i} dx \geq 	\sum_{i=1}^{N}C_i(\alpha, \gamma, p_i)^{p_i}\int_{\Omega} \Gamma^{\frac{\gamma}{2-\beta}}|X_i \Gamma^{\frac{1}{2-\beta}}|^{p_i} |u|^{p_i}dx,
		\end{equation}
		where $C_i(\alpha, \gamma, p_i):= \frac{(\alpha-1)(p_i-1)-\gamma-1}{p_i},\;p_i>1$, and $i=1,\ldots,N$.
\end{cor}
Note that we recover the result of D'Ambrosio in \cite[Theorem 2.7]{DAmbrosio_Hardy}.
Corollary \ref{cor0} is proved with the same approach as the previous case by considering the functions
\begin{equation*}
W_i =  \Gamma^{\frac{\gamma+p_i}{2-\beta}}\,\, \text{and} \,\, v = \Gamma^{-\frac{(\alpha-1)(p_i-1)-\gamma-1}{(2-\beta)p_i}}.
\end{equation*}
\begin{cor}\label{cor2}
		Let $\Omega \subset M$ be an admissible domain. Let $\alpha \in \mathbb{R}, \beta>2,$ $1<p_i<\beta+\alpha$ for $i=1,\ldots,N$. Then for all $u \in C_0^{\infty}(\Omega)$ we have
	\begin{equation}\label{5}
	\sum_{i=1}^{N}\int_{\Omega} \Gamma^{\frac{\alpha}{2-\beta}}|X_iu|^{p_i} dx \geq\sum_{i=1}^{N} C_i(\beta,\alpha,p_i)\int_{\Omega} \Gamma^{\frac{\alpha}{2-\beta}}\frac{|X_i\Gamma^{\frac{1}{2-\beta}}|^{p_i}}{\left(1+\Gamma^{\frac{p_i}{(p_i-1)(2-\beta)}}\right)^{p_i}}|u|^{p_i}dx,
	\end{equation}
	where $C_i(\beta,\alpha,p_i):=\left(\frac{\beta+\alpha-p_i}{p_i-1}\right)^{p_i-1}(\beta+\alpha)$.
\end{cor}
Note that a Carnot group version of inequality \eqref{5} was established by  Goldstein, Kombe and Yener in \cite{GKY_Hardy}. 
Corollary \ref{cor2} is proved with the same approach as the previous cases by considering the functions
\begin{equation*}
	W_i = \Gamma^{\frac{\alpha}{2-\beta}} \,\, \text{and} \,\, v = \left(1+\Gamma^{\frac{p_i}{(p_i-1)(2-\beta)}}\right)^{-\frac{\beta+\alpha-p_i}{p_i}}.
\end{equation*}
\begin{cor}\label{cor3}
	Let $\Omega \subset M$ be an admissible domain. Let $\alpha \in \mathbb{R}, \beta>2,$ $ 1<p_i<\beta+\alpha$ for $ i=1,\ldots,N$. Then for all $u \in C_0^{\infty}(\Omega)$ we have
	\begin{align}\label{2.25}
	\sum_{i=1}^{N}\int_{\Omega} &\left(1+\Gamma^{\frac{p_i}{(p_i-1)(2-\beta)}}\right)^{\alpha(p_i-1)}|X_iu|^{p_i} dx \\ &\geq\sum_{i=1}^{N} C_i(\beta,p_i,\alpha)\int_{\Omega} \frac{|X_i\Gamma^{\frac{1}{2-\beta}}|^{p_i}}{\left(1+\Gamma^{\frac{p_i}{(p_i-1)(2-\beta)}}\right)^{(1-p_i)(1-\alpha)}}|u|^{p_i}dx. \nonumber
	\end{align}
	where $C_i(\beta,p_i,\alpha):=\beta \left(\frac{p_i(\alpha-1)}{p_i-1}\right)^{p_i-1}$.
\end{cor}
Note that Carnot and Euclidean versions of inequality \eqref{2.25} were established in \cite{GKY_Hardy} and \cite{Skrt}, respectively. 
Corollary \ref{cor3} is proved with the same approach as the previous case by considering the functions 
\begin{equation*}
W_i = \left(1+\Gamma^{\frac{p_i}{(p_i-1)(2-\beta)}}\right)^{\alpha(p_i-1)}\,\, \text{and} \,\, v = \left(1+\Gamma^{\frac{p_i}{(p_i-1)(2-\beta)}}\right)^{1-\alpha}.
\end{equation*}
\begin{cor}\label{cor4}
	Let $\Omega \subset M$ be an admissible domain. Let $\beta>2$, $a,b>0$ and $\alpha,\gamma,m \in \mathbb{R}$. If $\alpha \gamma >0$ and $m \leq \frac{\beta-2}{2}$. Then for all $u \in C_0^{\infty}(\Omega)$ we have
	\begin{align}\label{2.26}
\int_{\Omega} \frac{(a+b\Gamma^{\frac{\alpha}{2-\beta}})^{\gamma}}{\Gamma^{\frac{2m}{2-\beta}}} |\nabla_X u|^{2} dx &\geq C(\beta,m)^2\int_{\Omega} \frac{(a+b\Gamma^{\frac{\alpha}{2-\beta}})^{\gamma}}{\Gamma^{\frac{2m+2}{2-\beta}}}|\nabla_X \Gamma^{\frac{1}{2-\beta}}|^2 |u|^{2}dx \nonumber \\
		&+ C(\beta,m)\alpha\gamma b\int_{\Omega} \frac{(a+b\Gamma^{\frac{\alpha}{2-\beta}})^{\gamma-1}}{\Gamma^{\frac{2m-\alpha+2}{2-\beta}}}|\nabla_X \Gamma^{\frac{1}{2-\beta}}|^2 |u|^{2}dx,
\end{align}
	where $C(\beta,m):=\frac{\beta-2m-2}{2}$ and $\nabla_X = (X_1,\ldots,X_N)$.
\end{cor}
Note that Carnot and Euclidean version of inequality \eqref{2.26} were established in \cite{GKY_Hardy} and \cite{GM_Bessel}, respectively. 
Corollary \ref{cor4} can be proved with the same approach for $p_i=2$, $i=1,\ldots,N,$ as the previous cases by considering the functions 
\begin{equation*}
W = \frac{(a+b\Gamma^{\frac{\alpha}{2-\beta}})^{\gamma}}{\Gamma^{\frac{2m}{2-\beta}}}\,\, \text{and} \,\, v = \Gamma^{-\frac{\beta-2m-2}{2(2-\beta)}}.
\end{equation*}
\subsection{Uncertainty principles for the vector fields} 
Theorem \ref{THMgeneralHardy} also implies the following uncertainty principles:
\begin{cor}\label{uncert1}
	Let $\Omega \subset M$ be an admissible domain. Let $\beta>2$. Then for all $u \in C_0^{\infty}(\Omega)$ we have
	\begin{equation}
	\frac{\beta^2}{4} \left(\int_{\Omega} |u|^2 dx\right)^2 \leq	\left( \int_{\Omega} |\nabla_X \Gamma^{\frac{1}{2-\beta}}|^{-2}|\nabla_X u|^2 dx \right)\left( \int_{\Omega} \Gamma^{\frac{2}{2-\beta}} |u|^2 dx \right). 
	\end{equation}
\end{cor}
\begin{proof}[Proof of Corollary \ref{uncert1}]
In Theorem \ref{THMgeneralHardy}, by letting 
 \begin{equation*}
 W(x)= |\nabla_X \Gamma^{\frac{1}{2-\beta}}|^{-2} \,\, \text{and} \,\, v = e^{-\alpha \Gamma^{\frac{2}{2-\beta}}},
 \end{equation*}
 where $ \alpha \in \mathbb{R}$, we arrive at 
 \begin{align*}
 - 4\alpha^2 \int_{\Omega}\Gamma^{\frac{2}{2-\beta}}|u|^2 dx + 
 2 \alpha \beta \int_{\Omega} |u|^2 dx -	\int_{\Omega} |\nabla_X \Gamma^{\frac{1}{2-\beta}}|^{-2}|\nabla_X u|^2 dx \leq 0.
 \end{align*}
 It can be noted that above inequality has the form $a\alpha^2 + b\alpha+c\leq0$ if we denote by 
 $$a :=- 4 \int_{\Omega}\Gamma^{\frac{2}{2-\beta}}|u|^2 dx,$$ $$b:=2  \beta \int_{\Omega} |u|^2 dx,$$ and 
 $$c:= -\int_{\Omega} |\nabla_X \Gamma^{\frac{1}{2-\beta}}|^{-2}|\nabla_X u|^2 dx.$$ Thus, we have $b^2-4ac\leq 0$ which proves Corollary \ref{uncert1}.
\end{proof}
\begin{cor}\label{uncert2}
	Let $\Omega \subset M$ be an admissible domain. Let $\beta>2$. Then for all $u \in C_0^{\infty}(\Omega)$ we have
	\begin{equation}
		\left( \int_{\Omega} |\nabla_X u|^2 dx \right)\left( \int_{\Omega} \Gamma^{\frac{2}{2-\beta}} |\nabla_X \Gamma^{\frac{1}{2-\beta}}|^2 |u|^2 dx\right) \geq \frac{\beta^2}{4} \left( \int_{\Omega} |\nabla_X \Gamma^{\frac{1}{2-\beta}}|^2 |u|^2 dx \right)^2.
	\end{equation}
\end{cor}
\begin{proof}[Proof of Corollary \ref{uncert2}]
	Setting
	\begin{equation*}
	W = 1 \,\, \text{and}\,\, v= e^{-\alpha \Gamma^{\frac{2}{2-\beta}}},
	\end{equation*}
	where $\alpha \in \mathbb{R}$, we have
	\begin{equation*}
	\int_{\Omega} |\nabla_X u|^2 dx \geq 2\alpha \beta \int_{\Omega}|\nabla_X \Gamma^{\frac{1}{2-\beta}}|^2 |u|^2 dx - 4\alpha^2 \int_{\Omega} \Gamma^{\frac{2}{2-\beta}} |\nabla_X \Gamma^{\frac{1}{2-\beta}}|^2 |u|^2 dx.
	\end{equation*}
	Using the same technique as before we prove Corollary \ref{uncert2}.
\end{proof}
\begin{cor}\label{uncert3}
	Let $\Omega \subset M$ be an admissible domain. Let $\beta>2$. Then for all $u \in C_0^{\infty}(\Omega)$ we have 
	\begin{align}
	\left( \int_{\Omega} |\nabla_X u|^2 dx \right)&\left( \int_{\Omega} \Gamma^{\frac{2}{2-\beta}} |\nabla_X \Gamma^{\frac{1}{2-\beta}}|^2 |u|^2 dx\right)\\
	&\geq \frac{(\beta-1)^2}{4} \left( \int_{\Omega} \Gamma^{-\frac{1}{2-\beta}}|\nabla_X \Gamma^{\frac{1}{2-\beta}}|^2 |u|^2 dx \right)^2. \nonumber
	\end{align}
\end{cor}
We can prove it with the same approach by considering the following pair
\begin{equation*}
W = 1 \,\, \text{and}\,\, v= e^{-\alpha \Gamma^{\frac{1}{2-\beta}}}.
\end{equation*}
\begin{rem}
Carnot group versions of these uncertainty principles were established in \cite{Kombe2010} and \cite{GKY_Hardy}.
\end{rem}
\section{Weighted anisotropic Rellich type inequalities}\label{sec1-2}
In this section, we now present the anisotropic (second order) Picone type identity. As a byproduct, we  obtain the weighted anisotropic Rellich type inequalities for the general vector fields.
\subsection{Anisotropic (second order) Picone type identity}
\begin{lem}\label{Picone1}
	Let $\Omega \subset \G$ be an open set. Let $u,v$ be twice differentiable a.e. in $\Omega$ and satisfying the following conditions: $u\geq0$, $v>0$, $X_i^2v<0$ a.e. in $\Omega$. Let $p_i>1$, $i=1,\ldots,N$. Then we have
	\begin{equation}\label{L=R}
		L_1(u,v) = R_1(u,v) \geq 0,
	\end{equation}
	where
	\begin{align*}
		R_1(u,v) := \sum_{i=1}^{N} |X_i^2 u|^{p_i} - \sum_{i=1}^{N} X_i^2 \left(\frac{u^{p_i}}{v^{p_i-1}}\right)|X_i^2 v|^{p_i-2} X_i^2 v,
	\end{align*}
	and
	\begin{align*}
	L_1(u,v) :=& \sum_{i=1}^{N} |X_i^2 u|^{p_i} - \sum_{i=1}^{N} p_i \left( \frac{u}{v}\right)^{p_i-1} X_i^2 u X_i^2 v |X_i^2 v|^{p_i-2} \\
	+& \sum_{i=1}^{N}(p_i-1)\left(\frac{u}{v}\right)^{p_i}|X_i^2 v|^{p_i}
	\\  -&\sum_{i=1}^{N} p_i(p_i-1) \frac{u^{p_i-2}}{v^{p_i-1}} |X_i^2 v|^{p_i-2} X_i^2 v \left(X_i u - \frac{u}{v}X_i v\right)^2.	
	\end{align*}
\end{lem}
\begin{proof}[Proof of Lemma \ref{Picone1}]
	A direct computation gives
	\begin{align*}
		X_i^2 \left(\frac{u^{p_i}}{v^{p_i-1}}\right) &= X_i \left( p_i \frac{u^{p_i-1}}{v^{p_i-1}}X_i u - (p_i-1)\frac{u^{p_i}}{v^{p_i}}X_i v \right) \\
		& = p_i(p_i-1)\frac{u^{p_i-2}}{v^{p_i-2}} \left(\frac{(X_iu) v - u(X_i v)}{v^2}\right)X_i u + p_i \frac{u^{p_i-1}}{v^{p_i-1}} X_i^2 u \\
		& -p_i(p_i-1)\frac{u^{p_i-1}}{v^{p_i-1}}\left(\frac{(X_iu) v - u(X_i v)}{v^2}\right)X_i v - (p_i-1)\frac{u^{p_i}}{v^{p_i}}X_i^2 v\\
		&= p_i(p_i-1) \left( \frac{u^{p_i-2}}{v^{p_i-1}} |X_i u|^2 - 2\frac{u^{p_i-1}}{v^{p_i}}X_i v X_i u +\frac{u^{p_i}}{v^{p_i+1}} |X_i v|^2  \right) \\
		& + p_i \frac{u^{p_i-1}}{v^{p_i-1}} X_i^2 u - (p_i-1)\frac{u^{p_i}}{v^{p_i}} X_i^2 v \\
		& = p_i(p_i-1) \frac{u^{p_i-2}}{v^{p_i-1}} \left(X_i u - \frac{u}{v} X_i v\right)^2  + p_i \frac{u^{p_i-1}}{v^{p_i-1}} X_i^2 u - (p_i-1)\frac{u^{p_i}}{v^{p_i}} X_i^2 v,
	\end{align*}
	which gives the equality in \eqref{L=R}. By Young's inequality we have
	\begin{equation*}
		\frac{u^{p_i-1}}{v^{p_i-1}} X_i^2 u X_i^2 v |X_i^2 v|^{p_i-2} \leq \frac{|X_i^2u|^{p_i}}{p_i} + \frac{1}{q_i}\frac{u^{p_i}}{v^{p_i}}|X_i^2 v|^{p_i},\quad i=1,\ldots,N,
	\end{equation*}
	where $p_i>1$ and $q_i>1$  with $\frac{1}{p_i}+\frac{1}{q_i}=1$. Since $X_i^2 v < 0,\, i=1,\ldots, N,$ we arrive at
	\begin{align*}
		L_1(u,v) &\geq \sum_{i=1}^{N} |X_i^2 u|^{p_i} +\sum_{i=1}^{N}(p_i-1) \frac{u^{p_i}}{v^{p_i}} |X_i^2 v|^{p_i} - \sum_{i=1}^{N} p_i\left(\frac{|X_i^2 u|^{p_i}}{p_i}+\frac{1}{q_i} \frac{u^{p_i}}{v^{p_i}}|X_i^2 v|^{p_i}\right) \\
		& -\sum_{i=1}^{N} p_i(p_i-1)\frac{u^{p_i-2}}{v^{p_i-1}} |X_i^2 v|^{p_i-2}X^2_iv \left|X_i u - \frac{u}{v} X_iv \right|^2 \\
		&=\sum_{i=1}^{N} \left(p_i-1- \frac{p_i}{q_i}\right)\frac{u^{p_i}}{v^{p_i}}|X_i^2 v|^{p_i}
		\\& -\sum_{i=1}^{N}p_i(p_i-1) \frac{u^{p_i-2}}{v^{p_i-1}} |X_i^2 v|^{p_i-2}X_i^2v \left|X_i u -\frac{u}{v}X_iv\right|^2
		\geq 0.
	\end{align*}
	This completes the proof of Lemma \ref{Picone1}.
\end{proof}
\subsection{Weighted anisotropic Rellich type inequalities}
\begin{thm}\label{Rellich}
			Let $\Omega \subset M$ be an admissible domain. Let $W_i(x) \in C^2(\Omega)$ and $H_i(x) \in L^1_{loc}(\Omega)$ be the nonnegative weight functions. Let $v > 0 $, $v \in C^{2}(\Omega) \bigcap C^1(\overline{\Omega})$ with
			\begin{equation}\label{123}
				X_i^2\left(W_i(x)|X_i^2v|^{p_i-2}X_i^2 v \right) \geq H_i(x)v^{p-1}, \,\, -X_i^2 v >0,
			\end{equation}
			a.e. in $\Omega$, for all $i=1,\ldots, N$. Then for every $0 \leq u \in C^{2}(\Omega) \bigcap C^1(\overline{\Omega})$ we have the following inequality  
	\begin{align}\label{rellich_ineq}
		\sum_{i=1}^{N}  \int_{\Omega} H_i(x) |u|^{p_i}dx &\leq \sum_{i=1}^{N} \int_{\Omega} W_i(x) |X_i^2 u|^{p_i} dx  \\
		& + \sum_{i=1}^{N}\int_{\partial \Omega} W_i(x)|X_i^2 v|^{p_i-2} X_i^2 v \langle \widetilde{\nabla}_i \left(\frac{u^{p_i}}{v^{p_i-1}}\right), dx\rangle \nonumber\\
		&- \sum_{i=1}^{N}\int_{\partial \Omega} \left(\frac{u^{p_i}}{v^{p_i-1}}\right) \langle \widetilde{\nabla}_i(W_i(x)|X_i^2 v|^{p_i-2} X_i^2 v),dx  \rangle,  \nonumber
	\end{align}
where $1<p_i<N$ for $i=1,\ldots,N$, and $\widetilde{\nabla}_iu=X_iuX_i$.
\end{thm}

	Note that a Carnot group version of Theorem \ref{Rellich} was obtained by  Goldstein, Kombe and Yener in \cite{GKY_Rellich}.  Moreover, it should be also noted that the function $v$ from the assumption \eqref{123} appears in the boundary terms \eqref{rellich_ineq}, which seems a new effect unlike known particular cases of Theorem \ref{Rellich}. 
\begin{proof}[Proof of Theorem \ref{Rellich}.]
	Let us give a brief outline of the following proof as in Theorem \ref{THMgeneralHardy}. We start by using the property of the anisotropic (second order) Picone type identity \eqref{L=R}, then we apply analogue of Green's second formula from Proposition \ref{Green} and the hypothesis \eqref{123}, respectively. Finally, we arrive at \eqref{rellich_ineq} by using $H_i(x)\geq 0$. Thus, we have
	\begin{align*}
		0 &\leq \int_{\Omega} W_i(x) L_1(u,v) dx = \int_{\Omega} W_i(x) R_1(u,v) dx \\
		& = \int_{\Omega} W_i(x) |X_i^2u|^{p_i} dx - \int_{\Omega}  X_i^2 \left(\frac{u^{p_i}}{v^{p_i-1}}\right) W_i(x)|X_i^2 v|^{p_i-2} X_i^2 v dx\\
		&= 	 \int_{\Omega} W_i(x) |X_i^2u|^{p_i} dx - \int_{\Omega} \frac{u^{p_i}}{v^{p_i-1}} X_i^2 \left(W_i(x) |X_i^2v|^{p_i-2}X_i^2v \right)dx \\
		& + \int_{\partial \Omega} \left(W_i(x)|X_i^2 v|^{p_i-2} X_i^2 v \langle \widetilde{\nabla}_i \left(\frac{u^{p_i}}{v^{p_i-1}}\right), dx\rangle - \left(\frac{u^{p_i}}{v^{p_i-1}}\right) \langle \widetilde{\nabla}_i (W_i(x)|X_i^2 v|^{p_i-2} X_i^2 v),dx  \rangle  \right)\\
		& \leq \int_{\Omega} W_i(x) |X_i^2u|^{p_i} dx - \int_{\Omega}
		 H_i(x)|u|^{p_i}dx\\
		 & + \int_{\partial \Omega} \left(W_i(x)|X_i^2 v|^{p_i-2} X_i^2 v \langle \widetilde{\nabla}_i \left(\frac{u^{p_i}}{v^{p_i-1}}\right), dx\rangle - \left(\frac{u^{p_i}}{v^{p_i-1}}\right) \langle \widetilde{\nabla}_i(W_i(x)|X_i^2 v|^{p_i-2} X_i^2 v),dx  \rangle  \right).		
	\end{align*}
	In the last line, we have used \eqref{123} which leads to \eqref{rellich_ineq}. 
\end{proof}
Let us recall that the operator $\mathcal{L}$ is the sum of squares of vector fields, defined by
\begin{equation}
	\mathcal{L} := \sum_{i=1}^{N} X_i^2.
\end{equation}
\begin{cor}\label{cor5}
		Let $\Omega \subset M$ be an admissible domain. Let $\beta>2$, $\alpha \in \mathbb{R}, \beta+\alpha>4$ and $\beta>\alpha$. Then for all $u \in C_0^{\infty}(\Omega \backslash \{0\})$ we have
	\begin{equation}
		\int_{\Omega} \frac{\Gamma^{\frac{\alpha}{2-\beta}}}{|\nabla_X \Gamma^{\frac{1}{2-\beta}}|^2}|\mathcal{L} u|^2 dx \geq C(\beta,\alpha) \int_{\Omega} \Gamma^{\frac{\alpha-4}{2-\beta}}|\nabla_X \Gamma^{\frac{1}{2-\beta}}|^2 |u|^2 dx,
	\end{equation}
	where $C(\beta,\alpha):=\frac{(\beta+\alpha-4)^2(\beta-\alpha)^2}{16}$ is in general sharp.
\end{cor}
\begin{rem}
	Note that Kombe \cite{Kombe2010} proved the sharpness of the constant appearing above inequality for the Carnot groups.
\end{rem}
\begin{proof}[The proof of Corollary \ref{cor5}]
	Let us choose the function $W(x)$ and $v$ such that
	\begin{equation}
	W(x) = \frac{\Gamma^{\frac{\alpha}{2-\beta}}}{|X_i\Gamma^{\frac{1}{2-\beta}}|^2} \,\, \text{ and } \,\, v= \Gamma^{\frac{\gamma}{2-\beta}},
	\end{equation}
	where $\gamma = -\frac{\beta+\alpha-4}{2}$. As in the case of the Hardy inequality, we use the notation $\Gamma = d^{2-\beta}$ for simplicity, then we get
	\begin{align*}
\sum_{i=1}^{N}	X_i^2 d^{\gamma} &=\sum_{i=1}^{N} X_i^2 \Gamma^{\frac{\gamma}{2-\beta}} = \sum_{i=1}^{N}X_i\left(\frac{\gamma}{2-\beta} \Gamma^{\frac{\gamma +\beta-2}{2-\beta}}X_i\Gamma \right) \\
	& = \frac{\gamma(\gamma+\beta-2)}{(2-\beta)^2}\Gamma^{\frac{\gamma +2\beta-4}{2-\beta}}\sum_{i=1}^{N}|X_i\Gamma|^2 + \frac{\gamma}{2-\beta} \Gamma^{\frac{\gamma +\beta-2}{2-\beta}}\sum_{i=1}^{N}X_i^2\Gamma\\
		& =\frac{\gamma(\gamma+\beta-2)}{(2-\beta)^2} d^{\gamma +2\beta-4}\sum_{i=1}^{N}|X_id^{2-\beta}|^2 \\
	& = \gamma(\gamma+\beta-2)d^{\gamma-2}\sum_{i=1}^{N}|X_id|^2.
	\end{align*}
	We observe that $\sum_{i=1}^{N}X_i^2 \Gamma=0$, since $\Gamma=\Gamma_y$ is the fundamental solution to $\mathcal{L}$. 
	Now we can compute the function $H(x)$,
	\begin{align*}
	X_i^2\left(W_i(x)X_i^2v\right) &= X_i^2 \left( \gamma(\gamma+\beta-2)d^{\gamma+\alpha-2}  \right) \\
	& = \gamma(\gamma+\beta-2) (\gamma+\alpha-2)(\gamma+\alpha+\beta-4)d^{\gamma+\alpha-4}|X_id|^2.
	\end{align*}
	By putting back $\gamma =  -\frac{\beta+\alpha-4}{2}$ we have 
	\begin{align*}
	\gamma+\beta-2 &= \frac{\beta-\alpha}{2},\\
	\gamma+\alpha-2 & = - \frac{\beta-\alpha}{2},\\
	\gamma+\alpha+\beta-4 & = \frac{\beta+\alpha-4}{2}.
	\end{align*}
	Then 
	\begin{align*}
	X_i^2\left(W(x)X_i^2v\right) &= \left(\frac{\beta+\alpha-4}{2}\right)^2 \left(\frac{\beta-\alpha}{2}\right)^2 d^{\alpha-4}|X_id|^2 v \\
	& =H(x) v.
	\end{align*}
	So we have the values of functions $W(x)$ and 
	$$H(x)=\left(\frac{\beta+\alpha-4}{2}\right)^2 \left(\frac{\beta-\alpha}{2}\right)^2 \Gamma^{\frac{\alpha-4}{2-\beta}}|X_i\Gamma^{\frac{1}{2-\beta}}|^2$$ 
	which allows to plug them in \eqref{rellich_ineq} yielding
	\begin{align*}
	\sum_{i=1}^{N}\left(\frac{\beta+\alpha-4}{2}\right)^2 \left(\frac{\beta-\alpha}{2}\right)^2 \int_{\Omega} \Gamma^{\frac{\alpha-4}{2-\beta}}|X_i\Gamma^{\frac{1}{2-\beta}}|^2|u|^2 dx \leq \sum_{i=1}^{N} \int_{\Omega} \frac{\Gamma^{\frac{\alpha}{2-\beta}}}{|X_i\Gamma^{\frac{1}{2-\beta}}|^2} |X_i^2 u|^2 dx.
	\end{align*}
	Note that the sharpness of the constant was obtained by Kombe \cite{Kombe2010} in the setting of the Carnot groups. In this general case, the argument is the same.
\end{proof}
The following corollary can be also proved with the same approach as the above case by setting
\begin{equation*}
W(x) = \frac{\Gamma^{\frac{\alpha+2p-2}{2-\beta}}}{|\nabla_X \Gamma^{\frac{1}{2-\beta}}|^{2p-2}} \,\, \text{and} \,\, v= \gamma^{-\frac{\beta+\alpha-2}{p(2-\beta)}}.
\end{equation*}
\begin{cor}\label{cor6}
	Let $\Omega \subset M$ be an admissible domain. Let $1<p<\infty$ and $2-\beta<\alpha<min\{(\beta-2)(p-1),(\beta-2)\}$. Then for all $u \in C_0^{\infty}(\Omega \backslash \{0\})$ we have
	\begin{equation}\label{3.7}
	\int_{\Omega}\frac{\Gamma^{\frac{\alpha+2p-2}{2-\beta}}}{|\nabla_X \Gamma^{\frac{1}{2-\beta}}|^{2p-2}} |\mathcal{L} u|^p dx \geq C(\beta,\alpha,p)^p \int_{\Omega} \Gamma^{\frac{\alpha-2}{2-\beta}}|\nabla_X \Gamma^{\frac{1}{2-\beta}}|^2 |u|^p dx,
	\end{equation}
	where $C(\beta,\alpha,p):=\frac{(\beta+\alpha-2)}{p} \frac{(\beta-2)(p-1)-\alpha}{p}$ is sharp.
\end{cor}\begin{rem}
Note that Lian \cite{Lian} presented the sharpness of the constant appearing in \eqref{3.7} in the case of the Carnot groups.
\end{rem}

\end{document}